\documentclass[11pt,oneside,reqno,british]{amsart}
\usepackage[T2A,T1]{fontenc}
\usepackage[latin9]{inputenc}
\usepackage{color}
\usepackage{babel}
\usepackage{mathtools}
\usepackage{amsthm}
\usepackage{amssymb}
\usepackage[unicode=true,pdfusetitle,
 bookmarks=true,bookmarksnumbered=false,bookmarksopen=false,
 breaklinks=true,pdfborder={0 0 0},pdfborderstyle={},backref=false,colorlinks=true]
 {hyperref}
\usepackage{breakurl}

\makeatletter

\DeclareRobustCommand{\cyrtext}{%
  \fontencoding{T2A}\selectfont\def\encodingdefault{T2A}}
\DeclareRobustCommand{\textcyr}[1]{\leavevmode{\cyrtext #1}}

\theoremstyle{plain}
\newtheorem{thm}{\protect\theoremname}
\theoremstyle{plain}
\newtheorem*{lem*}{\protect\lemmaname}
\theoremstyle{plain}
\newtheorem{lem}[thm]{\protect\lemmaname}
\theoremstyle{plain}
\newtheorem*{conjecture*}{\protect\conjecturename}

\usepackage{calrsfs}
\usepackage{graphicx}
\usepackage{color}
\usepackage{perpage}
\usepackage{upquote}
\usepackage{bbm}
\usepackage[shortlabels]{enumitem}
\usepackage{stmaryrd} 
\usepackage{extarrows} 
\usepackage{romanbar}

\MakePerPage{footnote}
\gdef\SetFigFontNFSS#1#2#3#4#5{} 
\gdef\SetFigFont#1#2#3#4#5{} 
\def\clap#1{\hbox to 0pt{\hss#1\hss}}

\DeclareMathOperator{\sign}{sign}

\definecolor{myblue}{rgb}{0.09,0.32,0.44} 
\hypersetup{pdfborder={0 0 0},pdfborderstyle={},colorlinks=true,linkcolor=myblue,citecolor=myblue,urlcolor=blue}

\theoremstyle{remark}
\newtheorem*{qst*}{Question}
\newtheorem*{rmrks*}{Remarks}
\newtheorem*{probs*}{Problems}
\theoremstyle{plain}
\newtheorem*{thm7p}{Theorem \ref{thm:KO98}'}

\newlength{\tempindent} 
\newcommand{\lazyenum}{
\setlength{\tempindent}{\parindent} 
\begin{enumerate}[leftmargin=0cm,itemindent=0.7cm,labelwidth=\itemindent,labelsep=0cm,align=left,label=\arabic*)]
\setlength{\parskip}{\smallskipamount}
\setlength{\parindent}{\tempindent}
}

\makeatletter
\def\moverlay{\mathpalette\mov@rlay}
\def\mov@rlay#1#2{\leavevmode\vtop{%
   \baselineskip\z@skip \lineskiplimit-\maxdimen
   \ialign{\hfil$\m@th#1##$\hfil\cr#2\crcr}}}
\newcommand{\charfusion}[3][\mathord]{
    #1{\ifx#1\mathop\vphantom{#2}\fi
        \mathpalette\mov@rlay{#2\cr#3}
      }
    \ifx#1\mathop\expandafter\displaylimits\fi}
\makeatother

\makeatletter
\ifdefined\andify
\renewcommand{\andify}{%
  \nxandlist{\unskip, }{\unskip{} \@@and~}{\unskip{} \@@and~}}
\def\author@andify{%
  \nxandlist {\unskip ,\penalty-1 \space\ignorespaces}%
    {\unskip {} \@@and~}%
    {\unskip \penalty-2 \space \@@and~}%
}
\let\@wraptoccontribs\wraptoccontribs
\@namedef{subjclassname@2020}{%
  \textup{2020} Mathematics Subject Classification}
\fi
\makeatother

\ifpdf
\def\afs#1#2{\href{#1}{\nolinkurl{#2}}}
\else
\def\afs#1#2{\burlalt{#1}{#2}}
\fi

\def\affs#1#2{\href{#1}{\nolinkurl{#2}}}

\makeatletter
\ifx\hyperlink\undefined
\newcommand{\customlabel}[2]{%
\protected@write \@auxout {}{\string \newlabel {#1}{{#2}{}}}}
\else
\newcommand{\customlabel}[2]{%
   \protected@write \@auxout {}{\string \newlabel {#1}{{#2}{\thepage}{#2}{#1}{}} }%
   \hypertarget{#1}%
}
\fi
\makeatother

\makeatother

\providecommand{\conjecturename}{Conjecture}
\providecommand{\lemmaname}{Lemma}
\providecommand{\theoremname}{Theorem}

\begin{document}
\title{Homeomorphisms and Fourier expansion}
\author{Gady Kozma and Alexander Olevski\u{i}}
\begin{abstract}
We survey our recent result that for every continuous function there
is an absolutely continuous homeomorphism such that the composition
has a uniformly converging Fourier expansion. We mention the history
of the problem, orginally stated by Luzin, and some details of the
proof.
\end{abstract}

\maketitle

\section{Introduction}

Can one improve convergence properties of Fourier series by change
of variable? The following theorem was proved by Julius P\'al in
1914 \cite{P14} in a slightly weaker form, and in the form below
by Harald Bohr in 1935 \cite{B35}. Salem gave a simplified proof
(\cite{S44} or \cite[\S 4.12, page 327]{B64}). See also \cite{KO23}.
\begin{thm}
For every real function $f\in C(\mathbb{T})$ there exists a homeomorphism
$h:\mathbb{T}\to\mathbb{T}$ such that the Fourier series of the superposition
$f\circ h$ converges uniformly.
\end{thm}

\begin{proof}
[Proof sketch]Let $f>0$. Consider the `star domain' $\Omega\coloneqq\{z=re^{i\theta},\ r<f(\theta)\}$
in $\mathbb{C}$. Let $F(z)$ be the conformal map from $D$ to $\Omega$.
Since the boundary of the domain $\Omega$ is a Jordan curve, by Caratheodory's
theorem, the map $F$ is a homeomorphism of the boundaries of $D$
and $\Omega$. In other words, $|F(e^{it})|=f(h(t))$ for some homeomorphism
$h$. The essence of the P\'al-Bohr Theorem is to show that $F$
belongs to the Sobolev space $W^{1/2,2}$ (half a derivative in $L^{2}$),
and so does $|F|$. Any continuous function in $W^{1/2,2}$ has a
uniformly converging Fourier series.
\end{proof}
The homeomorphism $h$ in this proof is, in general, singular. 

\begin{probs*}[N. N. Luzin]\leavevmode

\begin{enumerate}
\item Is it possible for any $f$ to find an \textit{absolutely continuous}
homeomorphism $h$ such that $f\circ h$ has a uniformly convergent
Fourier series?
\item Is it possible to find an $h$ so that the superposition has absolutely
convergent Fourier series? Equivalently, $f\circ h$ is in the Wiener
Algebra, i.e.\ in the space of all $f$ whose Fourier coefficients
are in $l^{1}(\mathbb{Z})$?
\end{enumerate}
\end{probs*}

Both problems first appeared in print in Nina Bari's book \emph{Trigonometric
series} \cite[page 330]{B64}. The second problem was resolved negatively
in 1981 \cite{O81}.
\begin{thm}
\label{thm:Luzin1}There exist a real function $f\in C(\mathbb{T})$
such that $f\circ h$ is not in the Wiener Algebra whenever $h$ is
any homeomorphism of $\mathbb{T}$. 
\end{thm}

Simultaneously, J. P. Kahane and Y. Katznelson \cite{KK81} proved
a complex version of the theorem namely construct a continuous $f:\mathbb{T}\to\mathbb{C}$
which cannot be brought into the Wiener Algebra by a homeomorphism.
Equivalently, there are two real functions $f$ and $g$ such that
for no homeomorphism $h$ one has that both $f\circ h$ and $g\circ h$
are in the Wiener algebra (thus the complex version is weaker).

Let us sketch the approach of \cite{O81}. The main lemma is 
\begin{lem*}
Suppose $F$ oscillates from $-1$ to $1$ $N$ times on an interval
$[0,\gamma]$. Namely, suppose
\[
F(t)=\sin(Ng(t))\quad\forall0\le t\le\gamma
\]
for some $g$ increasing on $[0,\gamma]$, $g(0)=0$, $g(\gamma)=1$.
And suppose $F$ vanishes outside $[0,\gamma]$. Then $||F||_{A}>K\log^{\alpha}N$
where $K$ and $\alpha>0$ are absolute constants.
\end{lem*}
Here $||F||_{A}=\sum|\widehat{F}(n)|$. It is instructive to consider
the lemma even for $g$ linear. One sees that what makes the $A$-norm
large is that the oscillations `end abruptly'. Had they been allowed
to reduce gradually to zero, the $A$ norm could have been made constant.

The proof \cite[Theorem 3.2]{O85} uses an inequality of Davenport.
Let us remark that Davenport proved his inequality as part of an improvement
of a result of Paul Cohen on Littlewood's conjecture on trigonometric
sums (incidently, that conjecture was proved at approximately the
same time, see \cite{K81,MPS81}).

\section{Singularity of the homeomorphism}

The original proof of the P\'al-Bohr theorem did not relate directly
to the singularity of the homeomorphism, as it came about as the boundary
value of a conformal map. It could be regular or singular, depending
on the function $f$ (certainly, it is singular in most cases of interest).

Contrariwise, there are proofs of the P\'al-Bohr theorem which do
not use conformal maps at all, and construct the homeomorphism using
the basic idea that a strongly singular homeomorphism can `push' all
the problem areas of the function $f$ into an area too small to be
noted by the Fourier expansion. We mention two such results, the first
due to A. Saakyan \cite{S79}.
\begin{thm}
\label{thm:Saakyan}Let $f$ be a real function in $C(\mathbb{T})$.
Then there is a homeomorphism $h$ such that $\widehat{f\circ h}(n)=o(1/n)$. 

Further, instead of $\{1/n\}$ one can take any decreasing sequence
$a(n)>n^{-3/2}$, which is not in $l^{1}$. 
\end{thm}

Note that the P\'al-Bohr theorem follows since a continuous function
with Fourier coefficients $o(1/n)$ has a uniformly converging Fourier
series. Further, even the $W^{1/2,2}$ formulation follows, since
if one takes $a(n)=1/n\log n$ (which satisfies the requirements of
the theorem) then one gets $\widehat{F}(n)=o(1/n\log n)$ which implies
$F\in W^{1/2,2}$.

The proof of Saakyan's theorem uses the classical Schauder basis of
triangle functions (with dyadic bases) in the space $C(\mathbb{T})$.
The core of the proof is showing that for a given $f$ one can construct
a homeomorphism $h$ so that the superposition $f\circ h$ has a `lacunary'
Schauder decomposition, in the sense that the support of a triangle
of a given dyadic length appears only once. It is easy to see that
this property implies the estimate of $\widehat{f\circ h}$. 

The second result we mention is due to Kahane and Katznelson \cite[Th\'eor\`eme 3]{KK83}.
\begin{thm}
\label{thm:KKKompact}For each compact set $K\subset C(\mathbb{T})$
there is a universal homeomorphism $h$ such that $f\circ h$ has
a uniformly convergent Fourier series, whenever $f$ is in $K$.
\end{thm}

The proof constructs $h$ such that $f\circ h$ is close to being
a constant on any interval in the complement of a Cantor set.

It is interesting to note that theorems \ref{thm:Saakyan} and \ref{thm:KKKompact}
cannot be combined. Indeed, a result of Kahane and Katznelson shows
that there exists a sequence $a$ not in $l^{1}$ such it is not even
possible to map two functions simultaneously into the space of functions
satisfying $\widehat{f}(n)=o(a(n))$. See \cite[Theorem 4.3]{O85}.

Returning to Luzin's second question, let us mention a result of Kahane
and Katznelson \cite[Exemple 5]{KK83} that shows that the homeomorphism
in the P\'al-Bohr theorem cannot be much better than absolutely continuous.
\begin{thm}
\label{thm:no C1+eps}For every $\varepsilon>0$ there exists a function
$f$ such that for any homeomorphism $h\in C^{1+\varepsilon}$ with
$h'(0)\ne0$ it is not the case that $f\circ h$ has a uniformly converging
(or even uniformly bounded) Fourier expansion.
\end{thm}

The example consists in constructing a sequence $a_{1}>a_{2}>\dotsb>0$
decreasing very fast, and taking the function to be on every interval
$[a_{k},ka_{k}]$ a sine wave with $k$ peaks. The function is zero
outside these intervals (see figure \ref{fig:KK}). One can then check
that a smooth $h$ does not distort the picture (for large $k$, i.e.\ small
intervals) sufficiently to remove the resonance of the sine wave with
the Dirichlet kernel $D_{r}$, only to move around the exact $r$
for which this resonance occurs. 
\begin{figure}
\includegraphics[width=12cm]{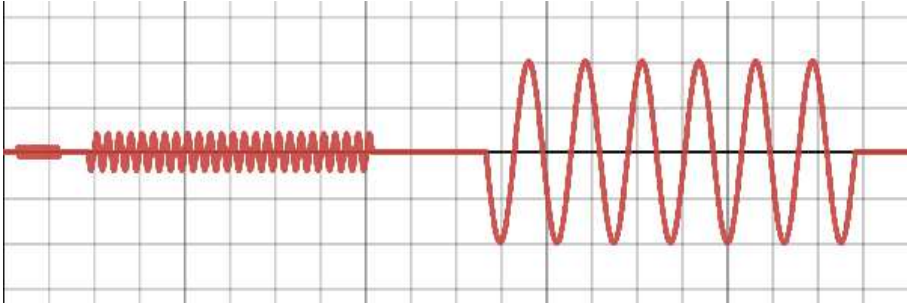} 

\caption{\label{fig:KK}The Kahane-Katznelson example.}
\end{figure}

We are now ready to present the main new result.
\begin{thm}
\label{thm:main}For every continuous real function $f$ there exists
an absolutely continuous homeomorphism $h$ such that Fourier series
of the superposition $f\circ h$ converges uniformly.
\end{thm}

Thus Theorem \ref{thm:main} resolves Luzin's first problem. The full
proof appears in \cite{KO23}. Here we will discuss informally some
ideas involved in it.

\section{Random Homeomorphisms}

Let us remind the reader our earlier result in the subject \cite{KO98}.
\begin{thm}
\label{thm:KO98}For any $f\in C(\mathbb{T})$ there exists a H\"older
homeomorphism $h$ such that \,$\|S_{n}(f\circ h)\|_{\infty}=o(\log\log n)$.
\end{thm}

The proof uses the Dubins-Freedman random homeomorphism \cite{DF67}.
This is a random increasing function $\phi:[0,1]\to[0,1]$ constructed
as follows. Take $\phi(0)=0$ and $\phi(1)=1$. Take $\phi(\frac{1}{2})$
to be uniform between 0 and 1. Then take $\phi(\frac{1}{4})$ to be
uniform between 0 and $\phi(\frac{1}{2})$, and $\phi(\frac{3}{4})$
to be uniform between $\phi(\frac{1}{2})$ and 1, and otherwise independent.
Continue similarly, taking $\phi(k/2^{n})$ to be uniform between
$\phi((k-1)/2^{n})$ and $\phi((k+1)/2^{n})$ for all odd $k\in\{1,3,\dotsc,2^{n}-1\}$.
Almost surely this can be extended to a H\"older homeomorphism \cite{GMW86}.
We may now restate Theorem \ref{thm:KO98}.

\begin{thm7p}\customlabel{thm:KOp}{\ref*{thm:KO98}'}For any continuous
function $f$, if $\phi$ is a Dubins-Freedman random homeomorphism
then the Fourier partial sums of the superposition $f\circ\phi$ have
norms $o(\log\log n)$ almost surely.

\end{thm7p}

Certainly, one would not expect a purely random construction to solve
Luzin's problem. In fact, the double logarithmic estimate above is
sharp (also in \cite{KO98}). Our hope in 1998 was to solve Luzin's
problem using a de Leeuw-Kahane-Katznelson like construction \cite{dLKK77},
with \cite{KO98} serving as an element of the construction, but this
hope never materialised.

\section{Random signs}

We shall now define and discuss a `linearised' version of the problem.
For a function $f$ with $\|f\|\leq1$ consider the partial sums:
\[
S_{n}(f\circ h;x)=\int f(h(t))D_{n}(x-t)\,dt,
\]
where $D_{n}$ is the Dirichlet kernel. To make the problem discrete
assume that $x=j/n$ for $j\in\{0,\dotsc,n-1\}.$

The linearisation we have in mind is to replace the homeomorphism
by a multiplication with signs. Instead of $f\circ h$ we replace
$f$ in the interval $[k/n,(k+1)/n)$ by $\epsilon_{k}f$ for some
$\epsilon_{k}=\pm1$.

\begin{qst*} Can one find $\epsilon_{k}$ such that
\[
\bigg|\sum_{k=0}^{n-1}\epsilon_{k}\int_{k/n}^{(k+1)/n}f(t)\,D_{n}\Big(\frac{j}{n}-t\Big)\,dt\bigg|<C\quad\forall j
\]
for a constant $C$ independent of $j$, $n$ and $f$?

\end{qst*}

The integrals above can be bounded, in absolute value, by $C/(|k-j|+1)$.
Hence the following lemma gives the answer.
\begin{lem}
\label{lem:Komlos}Let $v_{k,j}$ be numbers satisfying $|v_{k,j}|\leq1/(|k-j|+1)\,.$
Then there are signs $\epsilon_{k}$ such that 
\[
\bigg|\sum_{k=1}^{n}\epsilon_{k}v_{k,j}\bigg|<C\quad\forall j.
\]
\end{lem}

The proof of Theorem \ref{thm:main} in \cite{KO23} has an analogous
linearised lemma (lemma 3.1 in \cite{KO23}) but its formulation is
more complicated due to the need to control partial sums for all $n$
simultaneously. 

Nevertheless lemma \ref{lem:Komlos} above captures a significant
amount of the difficulty. For example, taking $\epsilon_{k}$ be i.i.d.~does
not work. The maximum becomes $\approx\log\log n$, which is also
the reason for the $\log\log n$ in Theorem \ref{thm:KOp} above.

For the proof of lemma \ref{lem:Komlos} we use a `hierarchical random
construction'. By this we mean that we divide the numbers $\{1,\dotsc,n\}$
to blocks of size $K$ (for some carefully chosen $K$) and use random
signs in each block. In each block it is possible to show that with
positive probability one can find signs which give a reasonably good
estimate, so we choose these signs. We then take the blocks, divide
them into larger scale blocks, and inside each `2nd level block' choose
signs randomly to get an improved estimate. Continuing with larger
and larger blocks we eventually reach the scale of $n$. This method
is inspired by the `renormalisation group' method in statistical mechanics
and field theory. We remark that a similar method was used by B. Kashin
\cite{K79} in his discrete version of the Menshov correction theorem.

It is interesting to compare Lemma \ref{lem:Komlos} with the the
following conjecture of J. Komlos.
\begin{conjecture*}
If $v_{k}$ are vectors in Euclidean space such that $\|v_{k}\|_{2}\leq1$
for all $k$ then there are signs $\epsilon_{k}$ such that 
\[
\Big\|\sum_{k}\epsilon_{k}v_{k}\Big\|_{\infty}<C.
\]
\end{conjecture*}
Had Koml\'os' conjecture been known, it would, of course, imply lemma
\ref{lem:Komlos}. J. Spenser proved that it is true if we allow $\epsilon_{k}\in[-1,1]$
and half of them in $\{-1,1\}$ \cite{S85}. See also \cite{B98}.

\section{\label{sec:Removing-randomness}Removing randomness}

Let us come back to random homeomorphisms. Here we describe the first
of the two main ideas involved in the proof of Theorem \ref{thm:main}.
Starting with a random homeomorphism of Dubins-Freedman type, we are
going to `remove randomness' step by step, keeping the behaviour of
the average $\mathbb{E}(f\circ\phi)$ and its partial Fourier sums
under control.

Fix a number $q$, $0<q<1$. The first step is to define a variation
of the Dubins-Freedman random homeomorphism $\psi_{q}$. The only
difference from Dubins-Freedman is as follows. Given a dyadic interval
$I=[(k-1)/2^{n},(k+1)/2^{n}]$, the image of the point $d=k/2^{n}$
is defined uniformly distributed on the interval concentric to $\psi(I)$
and of length $q|\psi(I)|$. 

Unlike the Dubins-Freedman homeomorphism which was H\"older almost
surely, the homeomorphism $\psi_{q}$ is Holder \emph{deterministically},
and for small $q$ the smoothness is close to 1. It is indeed rather
unusual in probabilistic constructions that the difference between
almost sure and deterministic behaviour makes any difference at all,
but here it does, as will be clear in a few paragraphs.

Assume a function $f$ is given with $\|f\|=1$. For dyadic points
$d=k/2^{n}$ with $k$ odd we call $n$ the rank of $d$. Let $n$
be fixed. Assume that we have a modification of $\psi_{q}$ (denote
it by $\phi$), satisfying the following conditions:
\begin{enumerate}
\item For all dyadic points $d$ of rank $<n$ we have that $\phi(d)$ is
not random.
\item For all $d$ of rank $n$ we have that $\phi(d)$ is uniformly distributed
over some interval $J(d)$ of length $q2^{-\ell}|\phi(I)|$, concentric
to $\phi(I)$. Here $\ell$ is some parameter independent of $d$.
\item For all $d$ of higher rank the conditional distributions of $\phi(d)$
remains as it was at the beginning, namely uniform on an interval
of length $q|\phi(I)|$, concentric to $\phi(I)$.
\end{enumerate}
We now make the next modification (denote it by $\phi'$), changing
only the condition (ii). It is replaced by
\begin{enumerate}
\item [(ii')]For all $d$ of rank $n$ the image $\phi'(d)$ is uniformly
distributed over a half of $J(d)$, upper or lower.
\end{enumerate}
The choice of the half, depending on $f$, is based on the linearised
lemma (Lemma \ref{lem:Komlos}, or to be more precise, its variation
in \cite{KO23} which we have not presented). To get a feeling of
which matrix we use lemma \ref{lem:Komlos} with, consider the matrix
$(v_{k,j})$ (a $2^{n-1}\times2^{n-1}$ matrix) where the $(k,j)$
entry is the effect of the choice of which half of $J(k2^{-n})$ we
take on $\mathbb{E}\int f(\phi'(t))D_{2^{n}}(j2^{-n}-t)$.

Note carefully how the problem was linearised. For each odd $k$ we
choose only between having $\phi'(k2^{-n})$ uniform in the upper
half of $J(k2^{-n})$ and uniform in the lower half of $J(k2^{-n})$.
In $\phi$ the value of $\phi(k2^{-n})$ is uniform on the whole of
$J(k2^{-n})$. Thus the distribution of $f\circ\phi$ is exactly the
average of the distributions of the two possibilities choices for
$f\circ\phi'$. Thus the two possible values of $\mathbb{E}\int(f\circ\phi')D$
average to $\mathbb{E}\int(f\circ\phi)D$. Of course, there are two
choices for each $k$, but the different intervals $[(k-1)2^{-n},(k+1)2^{-n}]$
behave independently. Thus 
\[
\int(f\circ\phi)D-\int(f\circ\phi')D=\sum\varepsilon_{k}v_{k}
\]
where $\varepsilon_{k}$ take the values $\pm1$, with $\varepsilon_{k}=1$
meaning we took the upper half of $J(k2^{-n})$ and $\varepsilon_{k}=-1$
meaning we took the lower half of $J(k2^{-n})$. The $v_{k}$ depend
also on the translation of the Dirichlet kernel i.e.\ on the $j$
in $D_{2^{n}}(j2^{-n}-t)$ --- above when we wrote $D$ we actually
meant $D_{2^{n}}(j2^{-n}-t)$. Thus we get a matrix $(v_{k,j})$ and
we are in a position to apply lemma \ref{lem:Komlos}.

In the full proof in \cite{KO23} we need to control $D_{r}$ for
all $r$ so we need many more rows --- the $k$ still take only $2^{n-1}$
values but we need to add to the rows that correspond to $D_{2^{n}}(j2^{-n}-t)$
also rows that correspond to $D_{r}$ for $r\ne2^{n}$. We get a $2^{n-1}\times\infty$
matrix but with the values decaying as the rows increase.

To sum up, we get that there exists a choice of a corresponding half
of $J(d)$ for each $d$ so that all Fourier sums of the function
\[
F=\mathbb{E}(f\circ\phi')-\mathbb{E}(f\circ\phi)
\]
will be small. More precisely, 
\[
\|S_{r}(F)(x)\|_{\infty}<\exp(-c(\ell+|n-\log_{2}r|))
\]
(recall that $\ell$ appeared in requirement (ii), where we had $|J(d)|=q2^{-\ell}|\phi(I)|$).

Now by induction over $\ell$ (for a fixed $n$) we get a random homeomorphism
$\phi_{n}$, for which all points of rank $n$ become non-random.
Then another induction, over $n$ provides a deterministic homeomorphism
$h$.

Summing all the deviations, we get a weaker version of Theorem \ref{thm:main}.
\begin{thm}
There is a H\"older homeomorphism $h$ (of any pre-given order $<1$)
such that $f\circ h$ has bounded Fourier sums.
\end{thm}

At this point the reader might be able to appreciate why it was important
that $\psi_{q}$, the random homeomorphism we started the process
with, was H\"older deterministically and not just almost everywhere.
Our process of moving from $\phi$ to $\phi'$ in a double induction
over $\ell$ and $n$ reduced the support of the homeomorphism gradually,
in a complicated and $f$-dependent manner. Thus the resulting homeomorphism,
$h$, is not a $\psi_{q}$-typical homeomorphism and it is impossible
to conclude from the fact that $\psi_{q}$ satisfies some property
with probability 1 anything about $h$. But \emph{all }homemorphisms
in the support of $\psi_{q}$ are H\"older, and thus so is $h$. 

We remark also that $h^{-1}$ (the inverse here is as homeomorphisms)
is also H\"older, again for any pre-given order $<1$. This is for
the same reason: the random homeomorphism we started with, $\psi_{q}$,
had this property deterministically.

\section{Absolutely continuous homeomorphisms}

The scheme above, started with $\psi_{q}$, where $q$ is a constant,
allows us to get a H\"older homeomorphism, but not an absolutely
continuous one.

To get an absolutely continuous homeomorphism we cannot have our random
homeomorphism fluctate at each scale identically. We need it to fluctuate
only where we really need it. For this purpose, we make $q$ not a
constant number but a function of dyadic rationals $d$, such that
the local `amount of randomness' would correspond to local oscillations
of $f$.

Roughly, if $f$ is `flat' on a dyadic interval centered at $d$,
then we do not need much randomness there and we can take the value
of $q(d)$ to be small. This makes the homeomorphism smoother which
allows us to hope eventually for to get it to be absolutely continuous. 

The problem, however, is whether any $f$ has enough `flatness'. On
the other hand, if $q(d)$ is too small then it provides not enough
randomness to kill resonances with the Dirichlet kernel. The Haar
decomposition of $f$ plays an important role in our construction.
Recall that the Haar functions are 
\[
1,\mathbbm{1}_{[0,1/2]}-\mathbbm{1}_{[1/2,1]},\frac{1}{\sqrt{2}}(\mathbbm{1}_{[0,1/4]}-\mathbbm{1}_{[1/4,1/2]}),\dotsc
\]
In other words, other than the function 1, each Haar function is supported
on some dyadic interval $I$ of order $n$ and takes the value $|I|^{-1/2}$
on its left half and $-|I|^{-1/2}$ on its right half. We denote this
function by $\chi_{I}$.

Let now $I$ be some dyadic interval and let $d$ by its middle. We
define
\[
q_{f}(d)\coloneqq\frac{1}{|I|^{3/2}}\sum_{J\subseteq I}\langle f,\chi_{J}\rangle^{2}|J|^{1/2}.
\]
To get a feeling for the function $q$, examine it for the case that
$f$ is a Radamacher function (a square wave, if you prefer) with
$2^{n}$ jumps, namely $f(x)=\sign(\sin(2^{n}\pi x))$. In this case
$\langle f,\chi_{J}\rangle$ is zero unless $|J|=2^{1-n}$, and in
this case $\langle f,\chi_{J}\rangle=|J|^{1/2}$. Inserting into the
formula above we see that $q(\frac{1}{2})$ is quite small ($2^{-(n-1)/2}$
to be precise) and then increases as the rank of $d$ increases, until
$d$ with rank $n$ where $q(d)=1$. For $d$ of higher rank, $q(d)=0$.
This corresponds with the following intuitive understanding how the
homeomorphism $\psi$ should be constructed. There's no reason to
let $\psi(\frac{1}{2})$ fluctuate too much (and indeed $q(\frac{1}{2})$
is small). But when reaching points of rank $n$ we certainly want
$\psi$ to have the possibility to fluctuate, and indeed there $q$
is large. (Of course, the square wave does not have any resonance
with the Dirichlet kernel to start with, because of its regular structure,
but the same considerations would hold for a perturbed square wave.
For example, take the lengths of all waves, not $2^{-n}$ like above,
but some arbitrary numbers between $2^{-n+1}$ and $2^{-n-1}$. The
Haar coefficients will be similar, and so will be $q$).

Once $q_{f}(d)$ is defined for all dyadic $d$ in $[0,1]$, we define
the `starting' homeomorphism $\psi_{q}$ as above. (For technical
reasons, it is in fact the inverse homeomorphism $\psi_{q}^{-1}$
that has a Dubins-Freedman structure. In other words, $q_{f}(d)$
does not determine how much $\psi$ will fluctuate at $d$ but at
$\psi^{-1}(d)$. The reader is referred to the end of section 4 in
\cite{KO23} for a detailed discussion of this point). 

We prove:
\begin{enumerate}
\item The random homeomorphism $\psi_{f}$ is absolutely continuous deterministically.
\item It admits a process of `removing of randomness', similar to one described
in section \ref{sec:Removing-randomness}, which finally gives a deterministic
homeomorphism $h$, as required in Theorem 2.
\end{enumerate}
The first claim requires $q$ to be `small enough', while the second
requires $q$ to be `large enough'. Thus the choice of $q$ is a delicate
point in the proof. In particular, in the proof of the first claim,
the \emph{John-Nirenberg inequality} for dyadic BMO functions is used
to control $q$.

\section{Some open problems}

\lazyenum 

\item Can one get a Lipschitz homeomorphism in Theorem \ref{thm:main}?
We remind the reader that this is impossible with a $C^{1+\varepsilon}$
homeomorphism, see Theorem \ref{thm:no C1+eps} above. It might also
be worth noting that in Theorem \ref{thm:main} it is possible to
require in addition that $h'\in L^{P}$ for any $p<\infty$, given
in advance. Further, the inverse homeomorphism $h^{-1}$ is also absolutely
continuous and also has $(h')^{-1}\in L^{p}$ for any $p<\infty$
(these generalisations are also proved in \cite{KO23}). But Lipschitz
would not follow from the same construction.

\item Is it possible to find for any $f\in C(\mathbb{T})$ a H\"older
homeomorphism $h$ such that $f\circ h$ is in $\widehat{l_{p}}$,
the space of functions whose Fourier transform is in $l_{p}(\mathbb{Z}$),
for some $p>1$? There are two strengthenings of this question which
are known to be false and one weakening which is known to be true.
First, it is not possible to have $f\circ h\in\widehat{l_{1}}$, even
without the requirement that $h$ be H\"older (see Theorem \ref{thm:Luzin1}
above). Next, it is not possible to have $f\circ h\in\widehat{l_{p}}$
with $h$ an absolutely continuous homeomorphism \cite[Theorem 5.1]{O85}.
Finally, it is possible to have $f\circ h\in\widehat{l_{p}}$ if one
does not care about the smoothness of $h$, see Theorem \ref{thm:Saakyan}
above.

Finally, it is interesting to see what the P\'al-Bohr approach gives.
On the one hand it gives $f\circ h\in W^{1/2,2}\subset\widehat{l_{p}}$
for all $p>1$. On the other hand, the homeomorphism $h$ does enjoy
some smoothness. It satisfies
\[
h(x)-h(y)\le\frac{C}{|\log(x-y)|}\qquad\forall x\ne y.
\]
Interestingly, the inverse homeomorphism $h^{-1}$ does not enjoy
any smoothness. For any $\omega$ with $\omega(x)\to0$ as $x\to0$
it is possible to construct a continuous function $f$ such that the
$h$ one gets from the P\'al-Bohr theorem does not satisfy $|h^{-1}(x)-h^{-1}(y)|\le C\omega(|x-y|)$.
We will not prove these two claims but we promise the readers that
they are relatively standard facts on the behaviour of the harmonic
measure.

\end{enumerate}

\end{document}